\documentclass[12 pt,twosided,reqno]{amsart}
\usepackage{amscd,amssymb}
\usepackage[arrow,matrix]{xy}
\usepackage{graphicx}
\usepackage{epsfig}
\theoremstyle{plain}
\newtheorem{theorem}[subsection]{Theorem}

\newtheorem{proposition}[subsection]{Proposition}
\newtheorem{cor}[subsection]{Corollary}

\theoremstyle{definition}

\newtheorem{definition}[subsection]{Definition}
\newtheorem{example}[subsection]{Example}

\newcommand{\PN}{\put(20,0){\line(-1,1){20}}\qbezier(0,0)(4,4)(8,8)\qbezier(20,20)(16,16)(12,12)}

\newcommand{\BOX}{ \multiput(0,0)(0,20){2}{\line(1,0){70}}\multiput(0,0)(70,0){2}{\line(0,1){20}}}

\def\Div{\hbox{Div}}

\begin{document}
\title [Fibonacci numbers and Positive braids]
{Fibonacci numbers and Positive braids }
\thanks{\emph{keywords and phrases:} Positive braids, square free braids, simple braids}
\thanks{This research is partially supported by Higher Education Commission, Pakistan.\\
2010 AMS classification: Primary 11B39, 20F36, 05A15; Secondary 05A05.}
\author{  REHANA ASHRAF$^{1}$,\,\,BARBU BERCEANU$^{1,2}$,\,\, AYESHA RIASAT$^{1}$}

\address{$^{1}$Abdus Salam School of Mathematical Sciences,
 GC University, Lahore-Pakistan.}
\email {rashraf@sms.edu.pk}
\email {ayesha.riyasat@gmail.com}
\address{$^{2}$
 Institute of Mathematics Simion Stoilow, Bucharest-Romania (permanent address).}
\email {Barbu.Berceanu@imar.ro}
  \maketitle
 \pagestyle{myheadings} \markboth{\centerline {\scriptsize
 REHANA ASHRAF,\,\,\,BARBU BERCEANU,\,\,\,AYESHA RIASAT   }} {\centerline {\scriptsize
 Fibonacci numbers and Positive braids}}
\begin{abstract}The paper contains enumerative combinatorics for positive braids, square free braids, and simple braids, emphasizing connections with classical Fibonacci sequence. The simple subgraph of the Cayley graph of the braid group is analyzed in the final part.
\end{abstract}
\section{\textbf{Introduction}}

The classical Fibonacci sequence, $(F_{n})_{n\geq 0}: 0,1,1,2,3,\ldots$ appears from time to time in enumerative questions related to Artin braids \cite{Artin}, the geometrical analogue of permutations. The \emph{positive $n$-braids} can be defined as words in the alphabet $\{x_{1},x_{2},\ldots,x_{n-1}\}$:
\begin{center}
\begin{picture}(130,50)
\thicklines
\multiput(15,10)(30,0){2}{\line(0,1){20}}
\put(70,10){\PN}
\multiput(115,10)(30,0){2}{\line(0,1){20}}
\put(-10,15){$x_{i}$}
\put(14,35){\tiny $1$}\put(34,35){\tiny$i-1$}\put(67,36){\tiny$i$}\put(83,36){\tiny$i+1$}
\put(110,35){\tiny$i+2$}\put(145,35){\tiny$n$}\multiput(25,17)(100,0){2}{$\cdots$}
\end{picture}
\end{center}
in which we identify two words obtained using finitely many changes of type
\begin{tabbing}
  0  \= 1 \= 2 \indent\indent\indent\indent\indent\indent\indent\indent\indent\indent\indent\indent\indent\indent\= 3 \kill
 \>  \>$\alpha(x_{i}x_{j})\beta\longleftrightarrow\alpha(x_{j}x_{i})\beta$ \> (for $|i-j|\geq2$) \\
 \> \> $\alpha(x_{i}x_{i+1}x_{i})\beta\longleftrightarrow\alpha(x_{i+1}x_{i}x_{i+1})\beta$\>  (for $i=1,2,\ldots,n-2$):
 \end{tabbing}
 
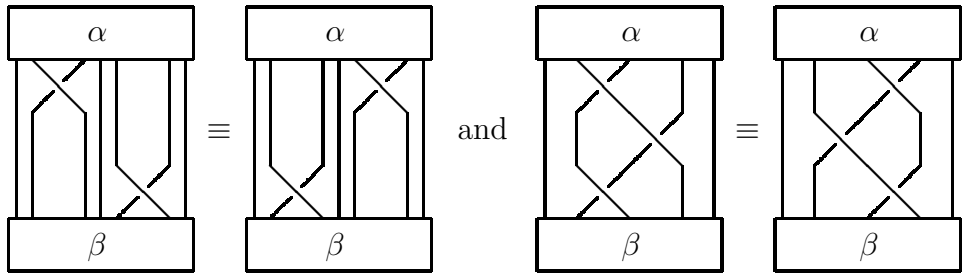
\begin{center}
\begin{picture}(350,120)
\thicklines
\multiput(0,0)(0,80){2}{\BOX}
\multiput(3,20)(32,0){3}{\line(0,1){60}}
\multiput(9,20)(20,0){2}{\line(0,1){40}}\put(9,60){\PN}
\multiput(41,80)(20,0){2}{\line(0,-1){40}}\put(41,20){\PN}

\multiput(90,0)(0,80){2}{\BOX}
 \multiput(93,20)(32,0){3}{\line(0,1){60}}
 \multiput(131,20)(20,0){2}{\line(0,1){40}}\put(131,60){\PN}
\multiput(99,80)(20,0){2}{\line(0,-1){40}}\put(99,20){\PN}

 \multiput(200,0)(0,80){2}{\BOX}
 \multiput(203,20)(64,0){2}{\line(0,1){60}}
 \put(215,20){\PN}\put(235,40){\PN}\put(215,60){\PN}
 \put(215,40){\line(0,1){20}}\put(255,20){\line(0,1){20}}\put(255,60){\line(0,1){20}}

\multiput(290,0)(0,80){2}{\BOX}
   \multiput(293,20)(64,0){2}{\line(0,1){60}}
\put(305,40){\PN}\put(325,20){\PN}\put(325,60){\PN}
\put(305,20){\line(0,1){20}}\put(305,60){\line(0,1){20}}
\put(345,40){\line(0,1){20}}
\multiput(75,50)(200,0){2}{$\equiv$}\put(170,50){and}
\multiput(30,87)(90,0){2}{$\alpha$}
\multiput(232,87)(90,0){2}{$\alpha$}
\multiput(30,6)(90,0){2}{$\beta$}
\multiput(232,6)(90,0){2}{$\beta$}

 \end{picture}
\end{center}
 \medskip

 A central role is played by \emph{Garside braid} \cite{Garside:69}: $\Delta_{n}=x_{1}(x_{2}x_{1})\ldots(x_{n-1}\ldots x_{1})$.
 
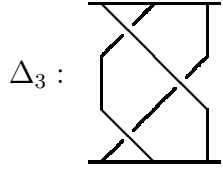
\begin{center}
\begin{picture}(100,80)
\thicklines
\multiput(10,10)(0,60){2}{\line(1,0){50}}
\put(15,10){\PN}\put(35,30){\PN}\put(15,50){\PN}
\put(15,30){\line(0,1){20}}\put(55,10){\line(0,1){20}}\put(55,50){\line(0,1){20}}
\put(-20,40){$\Delta_{3}:$}
 \end{picture}
\end{center}
 We will denote by $\mathcal{MB}_{n}$ the set of positive braids, by $\mathcal{MB}^{+}_{n}$ the set of positive braids not containing $\Delta_{n}$ as a subword, and by $\Div(\Delta_{n})$ the set of positive braids  which are subwords of Garside braid:
 $$\Div(\Delta_{n})=\{\omega \in \mathcal{MB}_{n} | \hbox{ there exist }\alpha,\beta\in\mathcal{MB}_{n}\hbox{ such that }\Delta_{n}=\alpha\omega\beta\}.$$
 A well known result says that the set of square free positive braids  coincides with $\Div(\Delta_{n})$. In \cite{Barbu-Rehana:02} is defined in many ways the set of \emph{simple braids}, $\mathcal{SB}_{n}\subset \Div{\Delta_{n}}$. One definition is:
 \begin{definition} A \emph{simple braid} is a positive braid $\beta\in \mathcal{MB}_{n}$ which contains a letter $x_{i}$ at most once.
 \end{definition}

In our first computations Fibonacci numbers $(F_{k})$ appear; $b_{k}$ and $b^{+}_{k}$ represents the number of braids of length $k$ in $\mathcal{MB}_{3}$ and $\mathcal{MB}^{+}_{3}$  respectively.
\begin{theorem}\label{gfpositive}The generating function of $\mathcal{MB}_{3}$ is

$$G_{\mathcal{MB}_{3}}(t)=\sum\limits_{k\geq 0}b_{k}t^{k}=1+2t+4t^{2}+7t^{3}+12t^{4}+20t^{5}+\ldots$$
where $b_{k}=F_{k+3}-1$, $k\geq 0$.
\end{theorem}
\begin{theorem}\label{gfprime}The generating function of $\mathcal{MB}^{+}_{3}$ is

$$G_{\mathcal{MB}^{+}_{3}}(t)=\sum\limits_{k\geq 0}b^{+}_{k}t^{k}=1+2t+4t^{2}+6t^{3}+10t^{4}+16t^{5}+\ldots$$
where $b^{+}_{k}=2F_{k-1}$, $k\geq 1$.
\end{theorem}
\begin{theorem}\label{sb and sp}
The number of simple braids in $\mathcal{SB}_{n}$ is $F_{2n-1}$.
\end{theorem}
The paper contains some other combinatorial problems related to positive braids.

\indent In the next section the proofs of the first two theorems are given.

\indent  In the third section the generating polynomial of the square free braids is computed (Proposition \ref{gf for square free}) and the recurrence relation for its coefficients are presented (Proposition \ref{recu pro for sf}).

\indent A proof of Theorem \ref{sb and sp}, the generating polynomial for simple braids, and some properties of its coefficients (Proposition \ref{formforsb}) are contained in section 4.

\indent The fifth section contains enumerative problems related to the set of conjugacy classes of simple braids (Proposition \ref{recu pro for cc}).

\indent In the last section we analyze the subgraph of the  Cayley graph of the braid group generated by simple braids (Proposition \ref{edges} and Proposition \ref{properties}).

\indent Connections between multiple Fibonacci-type recurrence \cite{nizami:08} and Jones polynomial and  Conway-Alexander polynomial for closed braids are presented in \cite{Barbu-Nizami:09} and \cite{Barbu:Rehana 1}.

\section{\textbf{Positive braids}}

The generating function for positive braids was computed by P. Deligne \cite{french} using invariants of Coxeter groups. A direct computation for $3$-braids was done by P. Xu \cite{Xu} and  an inductive algorithm for $G_{\mathcal{MB}_{n}}(t)$ and some generalizations are contained in Z. Iqbal \cite{Zaffar}. Using any of these references, we have
\begin{cor}\rm{(\cite{french}, \cite{Xu}, \cite{Zaffar})} The generating function for positive $3$-braids is given by
$$G_{\mathcal{MB}_{3}}(t)=\frac{1}{(1-t)(1-t-t^{2})}\cdot$$
\end{cor}
\emph{Proof of Theorem \ref{gfpositive}}  The expansion in simple parts $G_{\mathcal{MB}_{3}}(t)=\frac{2+t}{1-t-t^{2}}-\frac{1}{1-t}$
and the equality $(1-t-t^{2})^{-1}=\sum\limits_{k\geq 0}F_{k+1}t^{k}$ gives the result:

\indent\indent $b_{k}=(2F_{k+1}+F_{k})-1=(F_{k+1}+F_{k+2})-1=F_{k+3}-1.\,\,\,\,\,\,\,\,\,\,\,\,\,\,\indent\indent\indent\indent\indent\indent\indent\Box $
\emph{Proof of Theorem \ref{gfprime}}
Every positive braid $\beta$ can be written in a unique way as a product $\beta=\Delta^{k}_{n}\beta^{+}$ with $\beta^{+}\in\mathcal{MB}^{+}_{n}$ (see \cite{Garside:69}), therefore the decomposition $\mathcal{MB}_{3}=\small\coprod\small \limits _{k\geq 0}\Delta^{k}_{3}\cdot \mathcal{MB}^{+}_{3}$ implies:

$$G_{\mathcal{MB}^{+}_{3}}(t)=(1+t^{3}+t^{6}+\ldots )^{-1}\cdot G_{\mathcal{MB}_{3}}(t)=\frac{1+t+t^{2}}{1-t-t^{2}}=\sum\limits_{k\geq 1}b^{+}_{k}t^{k}.$$
Simple computations shows that $b^{+}_{0}=1$, $b^{+}_{1}=2=2F_{2}$, $b^{+}_{2}=4=2F_{3}$, and, for $k\geq 3$,  $b^{+}_{k}-b^{+}_{k-1}-b^{+}_{k-2}=0$, hence the result.$\indent\indent\indent\indent\indent\indent\indent\indent\indent\indent\indent\indent\indent\indent\indent\,\,\,\,\ \Box$

For a universal upper bound of the  growing type of $\mathcal{MB}_{n}$, see \cite{zaf:bar}.

\section{\textbf{Square free braids}}

To represent an element of $Div(\Delta_{n})$, i.e. a positive square free braid, we choose the canonical form given by the smallest elements in the length-lexicographic order (see \cite{Barbu}, \cite{Barbu:usman}):
$$\beta_{K,J}=\beta_{k_{1},j_{1}}\beta_{k_{2},j_{2}}\ldots \beta_{k_{s},j_{s}}$$
where $\beta_{k,j}=x_{k}x_{k-1}\ldots x_{j+1}x_{j}$,  $0\leq s\leq n-1,\,1\leq k _{1}<k_{2}\ldots <k_{s}\leq n-1,$
and $j_{h}\leq k_{h}$ for $h=1,\ldots ,s $ (the case $s=0$
corresponds to the unit $\beta=1$). For simplicity, we will write $Div_{n}$ for $Div(\Delta_{n})$. Let us denote by $d_{n,i}$  the number of divisors of $\Delta_{n}$ of length $i$ and by $G_{Div_{n}}(t)$ the generating polynomial  of the square
free $n$-braids.
\begin{proposition}\label{gf for square free}$G_{Div_{n}}(t)=\sum\limits_{i=0}^{n(n-1)/2} d_{n,i}t^{i}$

$\indent\indent\indent\indent\indent\indent\indent\indent\indent\indent\indent\,=(1+t)(1+t+t^{2})\ldots (1+t+t^{2}+\ldots +t^{n-1}).$
\end{proposition}
\begin{proof}We start the  induction  with $n=2$: $Div_{2}=\{1,x_{1}\}$ and $G_{Div_{2}}(t)=1+t$. The canonical form of square free braids shows that the map
$$f:Div_{n-1}\times \{1,\beta_{n-1,1},\beta_{n-1,2},\ldots,\beta_{n-1,n-1}\}\longrightarrow Div_{n},$$
defined by $f(\omega,1)=w$, $f(\omega, \beta_{n,k})=w\cdot \beta_{n,k}$, is a  bijection. The generating polynomial of the set $\{1,\beta_{n-1,k}\}_{k=1,\ldots,n-1}$ is $1+t+\ldots+t^{n-1}$, so
$G_{Div_{n}}(t)=G_{Div_{n-1}}(t)\cdot (1+t+\ldots+t^{n-1})$.
\end{proof}
\begin{cor}\label{recu pro for sf}The sequence $(d_{n,i})_{i=0,\ldots,\frac{n(n-1)}{2}}$ is symmetric and unimodular and satisfies the following recurrence relation:

\indent a) $d_{1,0}=1$, $d_{1,i}=0$ if $i\neq 1$;

\indent b) $d_{n+1,i}=d_{n,i}+d_{n,i-1}+\ldots+d_{n,i-n}$.
\end{cor}
\section{\textbf{Simple braids}}
The canonical form of a simple braids in $\mathcal{SB}_{n}$ is
$$\beta_{K,J}=\beta_{k_{1},j_{1}}\beta_{k_{2},j_{2}}\ldots
\beta_{k_{s},j_{s}}$$
where $1\leq k_{1}<k_{2}<\ldots k_{s}\leq
n-1$,  $j_{i}\leq k_{i}$ for all $i=1,2,\ldots,s$,  and also  $j_{i+1}>k_{i}$ for all
$i=1,2,\ldots,s-1$ (see \cite{Barbu-Rehana:02}).  Let us denote by $\mathcal{SB}^{i}_{n}$ the subset of simple braids of length $i$ in $\mathcal{SB}_{n}$ and $s_{n,i}$ its cardinality.
The generating polynomial  of simple $n$-braids is denoted by
$G_{\mathcal{SB}_{n}}(t)=\sum\limits_{i=0}^{n-1} s_{n,i}t^{i}$. We are interested to count the number of simple braids $G_{\mathcal{SB}_{n}}(1)$.

\begin{proposition}\label{formforsb} The sequence $(s_{n,i})$ is given by the recurrence:

\indent a) $s_{1,1}=0$ and $s_{1,i}=0$ for $i\neq 1$;

\indent b) $s_{n,i}=s_{n-1,i}+s_{n-1,i-1}+s_{n-2,i-2}+\ldots+s_{n-i,0}.$
\end{proposition}

\begin{proof}The set  $\mathcal{SB}^{i}_{n}$ can be decomposed as disjoint union  as follows:
  $$\mathcal{SB}^{i}_{n} = \mathcal{SB}^{i}_{n-1}\amalg( \mathcal{SB}^{i-1}_{n-1}\times\{x_{n-1}\})\amalg( \mathcal{SB}^{i-1}_{n-1}\times\{x_{n-1}x_{n-2}\})\amalg\ldots \amalg\{x_{n-1}\ldots x_{1}\}.$$
\end{proof}
\begin{cor}The sequence $(s_{n,i})$ satisfies also the recurrence : $$s_{n,i}=2s_{n-1,i-1}+s_{n-1,i}-s_{n-2,i-1}.$$
\end{cor}
\begin{example} Starting with $s_{n,0}=1$, $s_{n,1}=n-1$, and using the recurrence of Proposition \ref{formforsb}   we get $s_{n,2}=(n-1)(n+2)/2!$, $s_{n,3}=(n-3)(n+4)(n-1)/3!$ and $s_{n,4}=(n-4)(n+1)(n^{2}+5n-18)/4!$. Using the same recurrence we find that the last non zero coefficient is
$s_{n,n-1}=2^{n-2}$ (if $n\geq 2$). First values of $s_{n,i}$ (on the $n$-th line) are given in the triangle:
\end{example}

\begin{tabbing}
  0 \qquad\qquad\qquad\qquad\qquad\qquad\= 1 \= 2 \= 3 \= 4 \= 5 \= 6 \= 7 \= 8 \=9 \kill
  \> \>  \>  \>  \> 1 \>  \>  \>  \>  \\
  \> \>  \>  \> 1 \>  \> 1 \>  \>  \>  \\
  \> \>  \> 1 \>  \> 2 \>  \> 2 \>  \>  \\
  \> \> 1 \>  \> 3 \>  \> 5 \>  \> 4 \>  \\
 \> 1 \>  \> 4 \>  \> 9 \>  \> 12 \>  \> 8
\end{tabbing}


\begin{proposition}$s_{n,i}$ is a polynomial in $n$ of degree $i$ and leading coefficient is $1/i!$.
\end{proposition}
\begin{proof}The induction by $i$ starts with  $s_{n,0}=1$ and $s_{n,1}=n-1$. Using Proposition \ref{formforsb}, we have
$$s_{n,i}-s_{n-1,i}=s_{n-1,i-1}+s_{n-2,i-2}+\ldots+s_{n-i,0}$$
 where the sum is a polynomial in $n$ of degree $i-1$ and leading coefficient  $1/(i-1)!$. This implies that $s_{n,i}$ is a polynomial in $n$ of degree $i$ and leading coefficient is $1/i!$.

\end{proof}

\emph{Proof of Theorem \ref{sb and sp}}
By definition $G_{\mathcal{SB}_{n}}(1)=s_{n,0}+s_{n,1}+s_{n,2}+s_{n,n-2}+\ldots+s_{n,n-1}$. Using the recurrence given in Proposition \ref{formforsb}, we expand $G_{\mathcal{SB}_{n}}(1)=\sum\limits_{i=0}^{n-1}s_{n,i}$  and get $$G_{\mathcal{SB}_{n}}(1)=2G_{\mathcal{SB}_{n-1}}(1)+G_{\mathcal{SB}_{n-2}}(1)+G_{\mathcal{SB}_{n-3}}(1)+
\ldots+G_{\mathcal{SB}_{2}}(1)+G_{\mathcal{SB}_{1}}(1).$$
Starting an induction with  $G_{\mathcal{SB}_{1}}(1)=1=F_{1}=F_{2}$, $G_{\mathcal{SB}_{2}}(1)=2=F_{3}$, $G_{\mathcal{SB}_{3}}(1)=5=F_{5}$, we obtain
\begin{eqnarray*}
  G_{\mathcal{SB}_{n}}(1) &=& 2F_{2n-3}+F_{2n-5}+\ldots+F_{5}+F_{3}+F_{2} \\
  &=&2F_{2n-3}+F_{2n-5}+\ldots+F_{5}+F_{4}\\
  &=&\cdots\\
  &=&2F_{2n-3}+F_{2n-5}+F_{2n-6}=2F_{2n-3}+F_{2n-4}\\
  &=&F_{2n-2}+F_{2n-3}= F_{2n-1}.\indent\indent\indent\indent\indent\indent\indent\indent\indent\indent\indent\indent\indent\indent\indent\Box
 \end{eqnarray*}

\section{\textbf{Conjugacy classes of simple braids}}
A  simple braid $\beta\in \mathcal{SB}_{n}$ is conjugate to the braid
 $$\beta_{A}=(x_{1}x_{2}\ldots x_{s_{1}-1})(x_{s_{1}+1}\ldots x_{s_{2}-1})\ldots (x_{s_{r-1}+1}\ldots x_{s_{r}-1})$$
 (i.e. there is a positive braid $\alpha\in\mathcal{MB}_{n}$ such that $\beta\alpha=\alpha\beta_{A}$);  here $A=(a_{1},a_{2},\ldots,a_{r})$
is a sequence of integers satisfying
$a_{1}\geq a_{2}\geq\ldots \geq a_{r}\geq 2$ and $s_{i}=a_{1}+a_{2}+\ldots +a_{i}$. Conversely, if $\beta_{A}$ and $\beta_{A'}$ are conjugate, then the sequences $A$ and $A'$ coincide (see \cite{Barbu-Rehana:02}).

\indent The generating polynomial for distinct conjugate $n$-simple braid  is denoted by
 $Cf_{n}(t)=\sum^{n-1}\limits_{i=0}c_{n,i}t^{i}$, where $c_{n,i}$ is the  number of conjugacy classes of positive   simple braids of length $i$. A partition of a positive integer $m$ is a representation of $m$ in a form $m=m_{1}+m_{2}+\ldots+m_{k}$ where  the integers $m_{1},m_{2},\ldots,m_{k}$ satisfy the inequalities $m_{1}\geq m_{2}\geq\ldots\geq m_{k}\geq1$. The number of partitions of $m$ into $k$ parts is denoted by $P(m,k)$ (see \cite{Tomescue}).

 \begin{proposition}\label{recu pro for cc} The number of conjugacy classes of simple braids of length $i$ is given by
 $${\huge c} _{n,i}=P(i+min(i,n-i),min(i,n-i)).$$
  \end{proposition}
 \begin{proof}Consider $\beta_{A}=(x_{1}x_{2}\ldots x_{s_{1}-1})(x_{s_{1}+1}\ldots x_{s_{2}-1})\ldots (x_{s_{r-1}+1}\ldots x_{s_{r}-1})$, the canonical representative of a conjugacy class  in $\mathcal{SB}_{n}$, of length $i=s_{r}-r$.We associate to the sequence  $A=(a_{1}\geq a_{2}\geq \ldots a_{r})$ (here $a_{r}\geq 2$) the partition of $i$,  $i=(a_{1}-1)+(a_{2}-1)+\ldots(a_{r}-1)$. The condition $s_{r}=a_{1}+a_{2}+\ldots+a_{r}\leq n$
 implies $i+r \leq r$, therefore the number of conjugacy classes of simple braids of length $i$ is given by the number of partitions of $i$ into at most $n-i$ parts: $$c_{n,i}=P(i,1)+P(i,2)+\ldots+P(i, min(i,n-i)).$$ Using the relation  $P(n+k,k)=\sum\limits_{i=1}^{k}P(n,k)$ $(k \leq n)$ (see \cite{Tomescue}), we obtain the result.
\end{proof}

 \section{\textbf{Simple Graph}}

We consider the subgraph of the Cayley graph of the group $\mathcal{B}_{n}$ with vertices the  simple braids.
\begin{definition}The \emph{simple graph} $\Gamma_{\mathcal{S}\mathcal{B}_{n}}$ is the graph with vertices $\mathcal{SB}_{n}$ and edges between the simple braids $\beta\leftrightarrow \beta x_{i}$ $(i=1,\ldots, n-1)$. The number of vertices of $\Gamma_{\mathcal{S}\mathcal{B}_{n}}$ is $F_{2n-1}$.
\end{definition}
\begin{proposition}\label{edges}The  number of edges of the graph $\Gamma_{\mathcal{S}\mathcal{B}_{n}}$   is:
$$e(\Gamma_{\mathcal{SB}_{n}})=(n-1)s_{n,0}+(n-2)s_{n,1}+\ldots+s_{n,n-2}.$$
\end{proposition}
\begin{proof}The number of edges between a vertex $\beta$ of length $i$ and vertices of length $i+1$, $\beta x_{j}$, is the number of letters $x_{j}$ which are not in the simple braid $\beta$, and this number is $n-1-i$; this gives $(n-1-i)s_{n,i}$ edges between simple braids of length $i$ and $i+1$.
\end{proof}
\begin{proposition}\label{properties}a)  The graph $\Gamma_{\mathcal{SB}_{n}}$  is connected and  $n$-partite.

\indent b) $\Gamma_{\mathcal{SB}_{n}}$ is planar if and only if $n\leq 6.$
\end{proposition}
\begin{proof}a) Any vertex $\beta=x_{i_{1}}x_{i_{2}}\ldots x_{i_{k}}$ is connected to the empty word $1$ by the path $1$--- $x_{i_{1}}$---$ x_{i_{1}}x_{i_{2}}$--- $ \ldots$ ---$\beta$. The parts are given by simple braids of the same length. More generally, the Cayley graph of Coxeter groups are multipartite because multiplication by generators modifies the length of words with $\pm 1$.

\indent b) The graph $\Gamma_{\mathcal{SB}_{n}}$ is canonically embedded in $\Gamma_{\mathcal{SB}_{n+1}}$. The graph
$\Gamma_{\mathcal{SB}_{7}}$ contains $K_{3,3}$ as a subgraph and $\Gamma_{\mathcal{SB}_{6}}$ has a planar imbedding, see the next pictures:

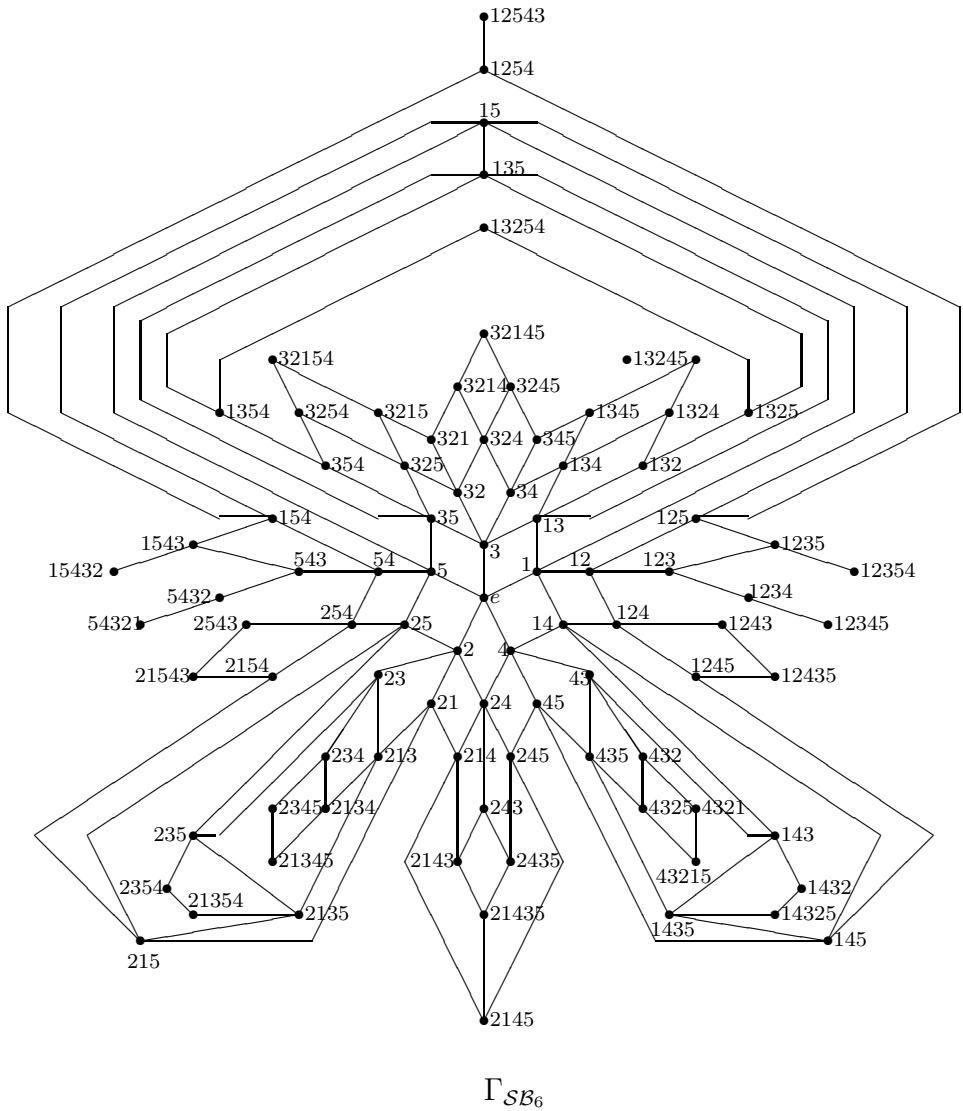
\begin{center}
\begin{picture}(400,440)
\put(198,188){\tiny$\bullet$\tiny$e$}
\put(218,198){\tiny$\bullet$}\put(214,200){\tiny$1$}
\put(188,168){\tiny$\bullet$\tiny$2$}
\put(198,208){\tiny$\bullet$}\put(202,205){\tiny$3$}
\put(208,168){\tiny$\bullet$}\put(205,168){\tiny$4$}
\put(178,198){\tiny$\bullet$\tiny$5$}
\put(238,198){\tiny$\bullet$}\put(232,202){\tiny$12$}
\put(178,148){\tiny$\bullet$\tiny$21$}
\put(218,218){\tiny$\bullet$}\put(222,215){\tiny$13$}
\put(158,159){\tiny$\bullet$}\put(162,157){\tiny$23$}
\put(188,228){\tiny$\bullet$\tiny$32$}
\put(228,178){\tiny$\bullet$}\put(218,178){\tiny$14$}
\put(198,148){\tiny$\bullet$\tiny$24$}
\put(208,228){\tiny$\bullet$\tiny$34$}
\put(238,159){\tiny$\bullet$}\put(232,157){\tiny$43$}
\put(198,368){\tiny$\bullet$}\put(198,372){\tiny$15$}
\put(168,178){\tiny$\bullet$\tiny$25$}
\put(178,218){\tiny$\bullet$\tiny$35$}
\put(218,148){\tiny$\bullet$\tiny$45$}
\put(158,198){\tiny$\bullet$}\put(158,202){\tiny$54$}
\put(268,198){\tiny$\bullet$}\put(260,202){\tiny$123$}
\put(258,238){\tiny$\bullet$\tiny$132$}
\put(158,128){\tiny$\bullet$\tiny$213$}
\put(178,248){\tiny$\bullet$\tiny$321$}
\put(248,178){\tiny$\bullet$}\put(250,182){\tiny$124$}
\put(228,238){\tiny$\bullet$\tiny$134$}
\put(138,128){\tiny$\bullet$\tiny$234$}
\put(198,248){\tiny$\bullet$\tiny$324$}
\put(188,128){\tiny$\bullet$\tiny$214$}
\put(308,98){\tiny$\bullet$\tiny$143$}
\put(198,108){\tiny$\bullet$\tiny$243$}
\put(258,128){\tiny$\bullet$\tiny$432$}
\put(278,218){\tiny$\bullet$}\put(265,218){\tiny$125$}
\put(88,98){\tiny$\bullet$}\put(75,98){\tiny$235$}
\put(118,218){\tiny$\bullet$\tiny$154$}
\put(218,248){\tiny$\bullet$\tiny$345$}
\put(128,198){\tiny$\bullet$}\put(128,202){\tiny$543$}
\put(68,58){\tiny$\bullet$}\put(65,50){\tiny$215$}
\put(168,238){\tiny$\bullet$\tiny$325$}
\put(208,128){\tiny$\bullet$\tiny$245$}
\put(138,238){\tiny$\bullet$\tiny$354$}
\put(198,348){\tiny$\bullet$}\put(203,350){\tiny$135$}
\put(328,58){\tiny$\bullet$\tiny$145$}
\put(148,178){\tiny$\bullet$}\put(138,182){\tiny$254$}
\put(238,128){\tiny$\bullet$\tiny$435$}
\put(288,178){\tiny$\bullet$\tiny$1243$}
\put(298,188){\tiny$\bullet$}\put(300,190){\tiny$1234$}
\put(268,258){\tiny$\bullet$\tiny$1324$}
\put(318,78){\tiny$\bullet$\tiny$1432$}
\put(138,108){\tiny$\bullet$\tiny$2134$}
\put(188,88){\tiny$\bullet$}\put(172,88){\tiny$2143$}
\put(188,268){\tiny$\bullet$\tiny$3214$}
\put(278,108){\tiny$\bullet$\tiny$4321$}
\put(308,208){\tiny$\bullet$\tiny$1235$}
\put(158,258){\tiny$\bullet$\tiny$3215$}
\put(118,108){\tiny$\bullet$\tiny$2345$}
\put(108,178){\tiny$\bullet$}\put(90,178){\tiny$2543$}
\put(128,258){\tiny$\bullet$\tiny$3254$}
\put(98,188){\tiny$\bullet$}\put(80,188){\tiny$5432$}
\put(78,78){\tiny$\bullet$}\put(62,78){\tiny$2354$}
\put(278,158){\tiny$\bullet$}\put(278,162){\tiny$1245$}
\put(198,388){\tiny$\bullet$\tiny$1254$}
\put(198,28){\tiny$\bullet$\tiny$2145$}
\put(118,158){\tiny$\bullet$}\put(102,162){\tiny$2154$}
\put(238,258){\tiny$\bullet$\tiny$1345$}
\put(98,258){\tiny$\bullet$\tiny$1354$}
\put(268,68){\tiny$\bullet$}\put(263,62){\tiny$1435$}
\put(88,208){\tiny$\bullet$}\put(70,208){\tiny$1543$}
\put(208,88){\tiny$\bullet$\tiny$2435$}
\put(208,268){\tiny$\bullet$\tiny$3245$}
\put(258,108){\tiny$\bullet$\tiny$4325$}
\put(298,258){\tiny$\bullet$\tiny$1325$}
\put(128,68){\tiny$\bullet$\tiny$2135$}
\put(328,178){\tiny$\bullet$\tiny$12345$}
\put(198,408){\tiny$\bullet$\tiny$12543$}
\put(198,328){\tiny$\bullet$\tiny$13254$}
\put(58,198){\tiny$\bullet$}\put(35,198){\tiny$15432$}
\put(88,68){\tiny$\bullet$}\put(88,73){\tiny$21354$}
\put(118,278){\tiny$\bullet$\tiny$32154$}
\put(68,178){\tiny$\bullet$}\put(50,178){\tiny$54321$}
\put(88,158){\tiny$\bullet$}\put(68,158){\tiny$21543$}
\put(308,158){\tiny$\bullet$\tiny$12435$}
\put(278,278){\tiny$\bullet$}\put(252,278){\tiny$\bullet$\tiny$13245$}
\put(308,68){\tiny$\bullet$\tiny$14325$}
\put(118,88){\tiny$\bullet$\tiny$21345$}
\put(198,288){\tiny$\bullet$\tiny$32145$}
\put(278,88){\tiny$\bullet$}\put(265,82){\tiny$43215$}
\put(198,68){\tiny$\bullet$\tiny$21435$}
\put(338,198){\tiny$\bullet$\tiny$12354$}
\put(180,370){\line(1,0){40}}
\put(180,350){\line(1,0){40}}
\put(100,221){\line(1,0){20}}
\put(160,221){\line(1,0){20}}
\put(220,221){\line(1,0){20}}
\put(280,221){\line(1,0){20}}
\put(130,200){\line(1,0){50}}
\put(220,200){\line(1,0){50}}
\put(110,180){\line(1,0){60}}
\put(230,180){\line(1,0){60}}
\put(90,160){\line(1,0){30}}
\put(280,160){\line(1,0){30}}
\put(90,70){\line(1,0){40}}
\put(270,70){\line(1,0){40}}
\put(70,60){\line(1,0){65}}
\put(265,60){\line(1,0){65}}
\multiput(20,260)(360,0){2}{\line(0,1){40}}
\multiput(40,260)(320,0){2}{\line(0,1){40}}
\multiput(60,260)(280,0){2}{\line(0,1){40}}
\multiput(70,265)(260,0){2}{\line(0,1){30}}
\multiput(80,270)(240,0){2}{\line(0,1){20}}
\multiput(100,260)(200,0){2}{\line(0,1){20}}
\multiput(180,200)(40,0){2}{\line(0,1){20}}
\put(200,190){\line(0,1){20}}
\multiput(160,130)(80,0){2}{\line(0,1){30}}
\multiput(140,110)(120,0){2}{\line(0,1){20}}
\multiput(120,90)(160,0){2}{\line(0,1){20}}
\multiput(190,90)(20,0){2}{\line(0,1){40}}
\multiput(200,30)(0,80){2}{\line(0,1){40}}
\put(200,390){\line(0,1){20}}
\put(200,350){\line(0,1){20}}

\put(220,250){\line(2,1){60}}\put(180,250){\line(-2,1){60}}
\put(210,230){\line(2,1){60}}\put(190,230){\line(-2,1){60}}
\put(200,210){\line(2,1){120}}\put(200,210){\line(-2,1){120}}
\put(240,220){\line(2,1){90}}\put(160,220){\line(-2,1){90}}
\put(220,200){\line(2,1){120}}\put(180,200){\line(-2,1){120}}
\put(240,200){\line(2,1){120}}\put(160,200){\line(-2,1){120}}
\put(300,220){\line(2,1){80}}\put(100,220){\line(-2,1){80}}
\put(300,280){\line(-2,1){100}}\put(100,280){\line(2,1){100}}
\put(320,290){\line(-2,1){120}}\put(80,290){\line(2,1){120}}
\put(330,295){\line(-2,1){110}}\put(70,295){\line(2,1){110}}
\put(340,300){\line(-2,1){140}}\put(60,300){\line(2,1){140}}
\put(360,300){\line(-2,1){140}}\put(40,300){\line(2,1){140}}
\put(380,300){\line(-2,1){180}}\put(20,300){\line(2,1){180}}

\put(260,240){\line(1,2){20}}
\put(220,220){\line(1,2){20}}
\put(200,210){\line(1,2){20}}
\put(190,230){\line(1,2){20}}
\put(180,250){\line(1,2){20}}
\put(220,250){\line(-1,2){20}}
\put(210,230){\line(-1,2){20}}
\put(200,210){\line(-1,2){20}}
\put(180,220){\line(-1,2){20}}
\put(140,240){\line(-1,2){20}}

\put(200,190){\line(2,1){20}}
\put(200,190){\line(-2,1){20}}
\put(200,190){\line(-1,-2){65}}
\put(200,190){\line(1,-2){65}}
\put(150,180){\line(-3,-2){120}}
\put(250,180){\line(3,-2){120}}
\put(170,180){\line(-3,-2){120}}
\put(230,180){\line(3,-2){120}}
\put(170,180){\line(-1,-1){80}}
\put(230,180){\line(1,-1){80}}
\put(190,170){\line(-4,-1){30}}
\put(210,170){\line(4,-1){30}}
\put(160,160){\line(-1,-1){60}}
\put(240,160){\line(1,-1){60}}
\put(160,160){\line(-2,-3){20}}
\put(240,160){\line(2,-3){20}}
\put(140,130){\line(-1,-1){20}}
\put(260,130){\line(1,-1){20}}
\put(180,150){\line(-1,-1){60}}
\put(220,150){\line(1,-1){60}}
\put(240,130){\line(1,-2){30}}
\put(200,150){\line(-1,-2){30}}
\put(200,150){\line(1,-2){30}}
\put(200,110){\line(-1,-2){10}}
\put(200,110){\line(1,-2){10}}
\put(190,90){\line(1,-2){10}}
\put(210,90){\line(-1,-2){10}}
\put(170,90){\line(1,-2){30}}
\put(230,90){\line(-1,-2){30}}
\put(180,150){\line(1,-2){10}}
\put(220,150){\line(-1,-2){10}}
\put(30,100){\line(1,-1){40}}
\put(370,100){\line(-1,-1){40}}
\put(50,100){\line(1,-2){20}}
\put(350,100){\line(-1,-2){20}}
\put(90,100){\line(4,-3){40}}\qbezier(90,100)(93,100)(98,100)
\put(310,100){\line(-4,-3){40}}\qbezier(300,100)(303,100)(308,100)
\put(90,100){\line(-1,-2){10}}
\put(310,100){\line(1,-2){10}}
\put(80,80){\line(1,-1){10}}
\put(320,80){\line(-1,-1){10}}
\put(190,170){\line(-2,1){20}}
\put(210,170){\line(2,1){20}}
\multiput(160,200)(20,0){2}{\line(-1,-2){10}}
\multiput(220,200)(20,0){2}{\line(1,-2){10}}
\multiput(120,220)(10,-20){2}{\line(-3,-1){60}}
\multiput(280,220)(-10,-20){2}{\line(3,-1){60}}
\put(110,180){\line(-1,-1){20}}
\put(290,180){\line(1,-1){20}}
\put(90,210){\line(4,-1){40}}
\put(310,210){\line(-4,-1){40}}
\put(190,170){\line(1,-2){10}}
\put(210,170){\line(-1,-2){10}}
\put(130,70){\line(-6,-1){60}}
\put(270,70){\line(6,-1){60}}
\put(160,130){\line(-1,-2){30}}
\put(200,0){\textbf{$\Gamma_{\mathcal{SB}_{6}}$}}

\end{picture}
\end{center}
\begin{center}
\begin{picture}(230,100)
\multiput(10,10)(100,0){3}{\line(0,1){80}}
\put(10,10){\line(5,4){100}}\put(10,10){\line(5,2){200}}
\put(110,10){\line(5,4){100}}\put(110,10){\line(-5,4){100}}
\put(210,10){\line(-5,4){100}}\put(210,10){\line(-5,2){200}}
\put(8,8){\tiny$\bullet$\tiny$e$}
\put(8,88){\tiny$\bullet$\tiny$1$}
\put(108,8){\tiny$\bullet$\tiny$136$}
\put(108,88){\tiny$\bullet$\tiny$3$}
\put(208,8){\tiny$\bullet$\tiny$26$}
\put(208,88){\tiny$\bullet$\tiny$6$}
\put(38,64){\tiny$\bullet$\tiny$13$}
\put(110,23){\tiny$\bullet$\tiny$36$}
\put(188,71){\tiny$\bullet$\tiny$16$}
\put(60,68){\tiny$\bullet$\tiny$14$}
\put(90,56){\tiny$\bullet$\tiny$4$}
\put(130,40){\tiny$\bullet$\tiny$24$}
\put(180,20){\tiny$\bullet$}\put(170,15){\tiny$246$}
\put(125,75){\tiny$\bullet$\tiny$35$}
\put(150,55){\tiny$\bullet$\tiny$5$}
\put(170,39){\tiny$\bullet$\tiny$25$}
\put(195,20){\tiny$\bullet$\tiny$2$}
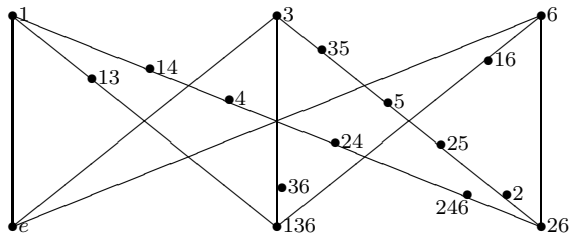
\put(60,-15){\textbf{A $K_{3,3}$ subgraph of $\Gamma_{\mathcal{SB}_{7}}$}}

\end{picture}
\end{center}

\end{proof}
The close relations between simple braids in $\mathcal{SB}_{n}$ and the corresponding permutations in  the symmetric group $\Sigma_{n}$ and also the simple part of the permutahedron (the simple graph $\Gamma_{\mathcal{SB}_{n}}$ is its one dimensional skeleton) are studied in \cite{?}.

\end{document}